\documentclass[11pt]{article}

\usepackage[T1]{fontenc}
\usepackage[margin=1in]{geometry}
\usepackage{amsmath,amssymb,amsthm,mathtools}
\usepackage[shortlabels]{enumitem}
\usepackage[hidelinks]{hyperref}
\usepackage[expansion=false]{microtype}
\usepackage{xcolor}

\hypersetup{
  pdftitle={Rank--Average-Degree Bound for Graph Energy},
  pdfauthor={Seyed Ahmad Mojallal},
  pdfsubject={Spectral graph theory and graph energy},
  pdfkeywords={graph energy, adjacency rank, average degree}
}

\newtheorem{theorem}{Theorem}[section]
\newtheorem{lemma}[theorem]{Lemma}
\newtheorem{proposition}[theorem]{Proposition}
\newtheorem{corollary}[theorem]{Corollary}
\theoremstyle{remark}

\newcommand{\E}{\mathcal E}
\newcommand{\dd}{\bar d}
\newcommand{\rank}{\operatorname{rank}}
\newcommand{\tr}{\operatorname{tr}}

\title{Rank--Average-Degree Bound for Graph Energy}

\author{
Seyed Ahmad Mojallal
}

\date{}

\begin{document}

\maketitle

\begin{abstract}
We prove that the energy $\E(G)$ of any simple graph $G$ of order $n\ge5$
satisfies
\[
  \E(G)\ge r(G)+\dd(G)-1,
\]
where $r(G)$ and $\dd(G)$ denote, respectively, the rank of the adjacency
matrix and the average degree of $G$.  We also characterize all extremal
graphs.  As consequences, our result settles five previously conjectured
lower bounds for the energy of nonsingular graphs in their stated ranges, namely
\[
 \begin{gathered}
 \E(G)\ge n-1+\dd(G),\qquad
 \E(G)\ge\Delta(G)+\delta(G),\qquad
 \E(G)\ge2\sqrt{\dd(G)(n-1)},\qquad\\
 \E(G)\ge\frac{M_1(G)}{m},\qquad
 \E(G)\ge\frac{M_1(G)}{2m}+\frac{2m}{n},
 \end{gathered}
\]
where $m$ is the number of edges, $\Delta(G)$ and $\delta(G)$ are the maximum and minimum
degrees, and the first Zagreb index
$ M_1(G)$ is the sum of degree squares.
\end{abstract}



\section{Introduction}

Let $G=(V(G),E(G))$ be a finite simple graph with
$|V(G)|=n$ and $|E(G)|=m$.  Let $A(G)$ be the adjacency matrix of $G$,
with eigenvalues $\lambda_1\ge\cdots\ge\lambda_n$.  Its energy,
introduced by Gutman \cite{Gutman1978}, is
\[
  \E(G)=\sum_{i=1}^n |\lambda_i|.
\]
When no confusion can arise, we write $A=A(G)$.  Write
$r(G)=\rank A(G)$ and $\dd(G)=\frac{2m}{n}$,
and, for $v\in V(G)$, write $d_v=d_G(v)$.  We denote the maximum and
minimum degrees by $\Delta(G)$ and $\delta(G)$, respectively.  When no
confusion can arise, we abbreviate $\dd(G)$ to $d$ inside proofs.  Akbari, Dabirian, and Ghasemi
\cite{AkbariDabirianGhasemi2022} conjectured that every nonsingular graph
satisfies
\[
  \E(G)\ge n-1+\dd(G),
\]
with exactly two exceptions of order four, namely $P_4$ and the paw graph.
Equivalently, the conjectured inequality holds for every nonsingular graph
of order $n\ge5$.
This conjecture strengthens an earlier conjecture of Akbari and
Hosseinzadeh \cite{AkbariHosseinzadeh2020}, who proposed that every nonsingular graph satisfies
\[
  \E(G)\ge \Delta(G)+\delta(G),
\]
with equality only for complete graphs.  Indeed,
\[
  n-1+\dd(G)\ge \Delta(G)+\delta(G),
\]
since $\Delta(G)\le n-1$ and $\delta(G)\le\dd(G)$.
More recently, Jahanbani and Gutman \cite{JahanbaniGutman2025} conjectured
that every nonsingular graph satisfies
\begin{equation}\label{eq:JG}
  \E(G)\ge 2\sqrt{\dd(G)(n-1)},
\end{equation}
with equality if and only if $G\cong K_n$.

Recent work shows that lower bounds for the energy of nonsingular graphs
remain an active theme.  Oboudi \cite{Oboudi2024} proved
\[
 \E(G)\ge n-1+\sqrt{2m-n+1}
\]
for graphs having no adjacency eigenvalue in $(-1,1)$, and Rakshith and
Yashaswini \cite{RakshithYashaswini2026} obtained further determinant- and
degree-based lower bounds and extensions of this spectral-gap setting.
In another direction, Das and Ghalavand \cite{DasGhalavand2025}
conjectured that every nonsingular graph satisfies
\[
 \E(G)\ge \frac{M_1(G)}{m}
\]
and
\[
 \E(G)\ge \frac{M_1(G)}{2m}+\frac{2m}{n},
\]
where
\[
 M_1(G):=\sum_{v\in V(G)}d_v^2
\]
is the first Zagreb index.  They further conjectured that equality in
either bound holds if and only if $G\cong K_n$.  As a consequence of
Theorem~\ref{thm:main}, both conjectures are settled; see
Corollary~\ref{cor:zagreb}.

The natural rank extension is the following inequality.

\begin{theorem}\label{thm:main}
Every finite simple graph $G$ of order $n\ge5$ satisfies
\begin{equation}\label{eq:main}
  \E(G)\ge r(G)+\dd(G)-1.
\end{equation}
\end{theorem}

\begin{theorem}[Equality characterization]\label{thm:equality}
Let $G$ be a finite simple graph of order $n\ge5$.  Then equality holds in
\eqref{eq:main} if and only if
\[
  G\cong K_n,
  \qquad\text{or}\qquad
  n\text{ is even and }G\cong \frac n2K_2.
\]
\end{theorem}

For a nonsingular graph of order $n\ge5$, Theorem~\ref{thm:main} gives
\[
 \E(G)\ge n-1+\dd(G).
\]
Since
\[
 n-1+\dd(G)\ge\Delta(G)+\delta(G)
\]
and, by AM--GM,
\[
 n-1+\dd(G)\ge2\sqrt{\dd(G)(n-1)},
\]
the theorem settles both earlier conjectured bounds as well.  Equality in
either consequence forces $G\cong K_n$. 

The two Zagreb-index conjectures are recorded explicitly in
Corollary~\ref{cor:zagreb}.

The rank formulation also yields structural lower bounds for graph energy.
For example, combining Theorem~\ref{thm:main} with standard lower bounds
for the adjacency rank gives bounds in terms of the clique number, induced
matching number, total domination number, and diameter; these consequences
are recorded in Corollary~\ref{cor:rank-consequences}.

The order restriction is sharp: $P_4$ is nonsingular and
$\E(P_4)=2\sqrt5<9/2=4+\dd(P_4)-1$.  The bound itself is sharp in two
distinct ways: complete graphs attain equality, and so do perfect matchings.
Theorem~\ref{thm:equality} shows that these are the only extremal graphs.

Two recent theorems are decisive.  Kumar and Pragada
\cite{KumarPragada2026} proved $\E(G)\ge2(n-\alpha(G))$, and Liu, Tang, and
Zhang \cite{LiuTangZhang2026} proved that each of the positive and negative
square energies of a connected graph is at least $n-1$.  We combine these
with two new ideas.  First, a vertex-energy estimate shows that
$\E(G)-\dd(G)$ does not increase under vertex deletion.  This turns a
nonsingular theorem into a rank theorem.  Second, determinant integrality and
the square-energy constraints reduce the nonsingular problem to a scalar
inequality.  The case of two positive eigenvalues, which lies outside that
relaxation, is treated through the sparse complement.

The scalar inequality is proved entirely analytically, without computer-assisted verification.

The paper is organized as follows.  Section~\ref{sec:prelim} collects the
spectral tools used throughout.  Section~\ref{sec:deletion} establishes the
vertex-energy deletion inequality and the order-five base case.
Section~\ref{sec:reduction} develops the determinant--variance reduction, and
Section~\ref{sec:scalar} proves the resulting scalar inequality.
Section~\ref{sec:p2} treats the exceptional case of two positive eigenvalues
through the sparse complement.  Section~\ref{sec:completion} completes the
proof of the rank bound, Section~\ref{sec:equality} determines all graphs
attaining equality, and Section~\ref{sec:consequences} records the Zagreb-index
and structural consequences.

\section{Preliminaries}\label{sec:prelim}

Let
\[
 p=n^+(G),\qquad q=n^-(G),\qquad \eta=n^0(G),
 \qquad r=p+q.
\]
We use
\[
 s^+(G)=\sum_{\lambda_i>0}\lambda_i^2,
 \qquad
 s^-(G)=\sum_{\lambda_i<0}\lambda_i^2.
\]
Since $\tr A=0$, the sum of the positive eigenvalues equals the sum of the
absolute values of the negative eigenvalues, and $\E(G)$ is twice either
sum.  Also
\begin{equation}\label{eq:squares}
 s^+(G)+s^-(G)=\tr A^2=n\dd(G).
\end{equation}

We shall use the following results.

\begin{enumerate}[(i)]
\item For every graph $G$,
\begin{equation}\label{eq:KP}
 \E(G)\ge2\bigl(n-\alpha(G)\bigr)
\end{equation}
by \cite{KumarPragada2026}.
\item If $G$ is connected, then
\begin{equation}\label{eq:LTZ}
 \min\{s^+(G),s^-(G)\}\ge n-1
\end{equation}
by \cite{LiuTangZhang2026}.
\item If $G$ is nonsingular of order at least five, the desired inequality
holds when either $\lambda_1(G)\le7.11$ or
$\dd(G)\le n-2\log n-3$.  These are Theorems 2.1 and 3.3 of
\cite{AkbariDabirianGhasemi2022}.
\end{enumerate}

We also use the standard inertia bounds
\begin{equation}\label{eq:inertia-alpha}
 p\le n-\alpha(G),\qquad q\le n-\alpha(G),
\end{equation}
which follow by interlacing an independent-set principal submatrix.  Standard
facts about interlacing, Ky Fan's principle, and matrix absolute values may be
found in \cite{Bhatia1997,BrouwerHaemers2012}.

For a square matrix $M$, we write
\[
 \det{}^* M:=\prod_{\lambda_i\ne0}\lambda_i
\]
for its pseudodeterminant, that is, the product of its nonzero eigenvalues.
The next observation will be used repeatedly.

\begin{lemma}\label{lem:pseudodet}
Let $M$ be a symmetric integral matrix of rank $r$.  Then
$\det{}^*M$ is a nonzero integer.  In particular, if $M=A(G)$, then
\begin{equation}\label{eq:E-rank}
 \E(G)\ge r.
\end{equation}
Moreover, $M$ has a nonsingular principal submatrix of order $r$.
\end{lemma}
\begin{proof}
If $r=0$, the assertions are immediate.  Assume $r\ge1$.
Since $M$ has rank $r$, its characteristic polynomial has the form
\[
 \chi_M(x)
 =
 x^{n-r}\prod_{i=1}^r(x-\lambda_i),
\]
where $\lambda_1,\ldots,\lambda_r$ are the nonzero eigenvalues of $M$.
Hence the coefficient of $x^{n-r}$ is, up to sign,
$\prod_{i=1}^r\lambda_i$.
By the principal-minor expansion of the characteristic polynomial, the same
coefficient is, up to sign, the sum of all principal minors of order $r$.
It is therefore a nonzero integer, and at least one of those principal
minors is nonzero.  In particular,
\[
 \left|\prod_{i=1}^r\lambda_i\right|\ge1.
\]
Finally, AM--GM applied to the absolute values of the $r$ nonzero
eigenvalues gives \eqref{eq:E-rank}.
\end{proof}

\section{Vertex energy and rank-preserving deletion}\label{sec:deletion}

For a vertex $v$, define its vertex energy by
\[
 e_v=(|A|)_{vv},\qquad |A|=(A^2)^{1/2}.
\]

\begin{lemma}\label{lem:vertex-energy}
Let $G$ be a graph of order $n\ge5$ and average degree
$d=\dd(G)$. Then, for every $v\in V(G)$,
\begin{equation}\label{eq:vertex-energy}
  e_v\ge \frac{2d_v-d}{n-1},
\end{equation}
where $d_v=d_G(v)$.
\end{lemma}

\begin{proof}
The assertion is immediate if $G$ is empty.  Assume therefore that $G$
has at least one edge, and let $\rho$ and $-\sigma$ be the largest and smallest
eigenvalues of its adjacency matrix $A$, respectively.  Thus
$\rho,\sigma>0$ and the spectrum of $A$ is contained in
$[-\sigma,\rho]$.

For every $x\in[-\sigma,\rho]$,
\begin{equation}\label{eq:scalar-vertex}
 |x|
 \ge
 \frac{2x^2+(\sigma-\rho)x}{\rho+\sigma}.
\end{equation}
Indeed, for $x\ge0$, the difference between the left- and right-hand sides
is
\[
 \frac{2x(\rho-x)}{\rho+\sigma}\ge0,
\]
whereas for $x\le0$ it is
\[
 \frac{2(-x)(\sigma+x)}{\rho+\sigma}\ge0.
\]

Since $A$ is real symmetric, applying \eqref{eq:scalar-vertex} to each
eigenvalue of $A$ by spectral functional calculus gives
\[
 |A|
 \succeq
 \frac{2A^2+(\sigma-\rho)A}{\rho+\sigma},
\]
where $|A|=(A^2)^{1/2}$.  Taking the $(v,v)$ entry and using
$A_{vv}=0$ yields
\begin{equation}\label{eq:local-spread}
 e_v=(|A|)_{vv}
 \ge
 \frac{2(A^2)_{vv}}{\rho+\sigma}
 =
 \frac{2d_v}{\rho+\sigma},
\end{equation}
since
\[
 (A^2)_{vv}=d_v.
\]

Moreover,
\[
 \rho^2+\sigma^2
 \le
 \sum_{i=1}^n\lambda_i^2
 =
 \operatorname{tr}(A^2)
 =
 nd.
\]
Consequently,
\begin{equation}\label{eq:spread}
 (\rho+\sigma)^2
 \le
 2(\rho^2+\sigma^2)
 \le
 2nd,
\end{equation}
and hence
\[
 \rho+\sigma\le\sqrt{2nd}.
\]

If $2d_v\le d$, then the right-hand side of
\eqref{eq:vertex-energy} is nonpositive, while $e_v\ge0$, so there is
nothing to prove.  Suppose therefore that $2d_v>d$, and set
\[
 y:=\frac{2d_v}{d}>1.
\]
By \eqref{eq:local-spread} and \eqref{eq:spread},
\[
 e_v
 \ge
 \frac{2d_v}{\sqrt{2nd}}
 =
 y\sqrt{\frac{d}{2n}}.
\]
Thus it remains to show
\[
 y\sqrt{\frac{d}{2n}}
 \ge
 \frac{d(y-1)}{n-1},
\]
because
\[
 \frac{2d_v-d}{n-1}
 =
 \frac{d(y-1)}{n-1}.
\]

Since $d_v\le n-1$ and $d_v=yd/2$, we have
\[
 d\le\frac{2(n-1)}{y}.
\]
Both sides of the desired inequality are positive, so after squaring it is
enough to prove
\[
 (n-1)^2y^2
 \ge
 2nd(y-1)^2.
\]
Using the preceding upper bound on $d$, it is sufficient that
\[
 (n-1)y^3\ge4n(y-1)^2,
\]
or equivalently,
\[
 \frac{y^3}{(y-1)^2}\ge\frac{4n}{n-1}.
\]
Now
\[
 f(y):=\frac{y^3}{(y-1)^2},
 \qquad y>1,
\]
satisfies
\[
 \frac{f'(y)}{f(y)}
 =
 \frac{3}{y}-\frac{2}{y-1}
 =
 \frac{y-3}{y(y-1)}.
\]
Hence $f$ attains its minimum at $y=3$, and
\[
 \min_{y>1}\frac{y^3}{(y-1)^2}
 =
 \frac{27}{4}.
\]
Finally, for $n\ge5$,
\[
 \frac{27}{4}
 >
 \frac{4n}{n-1}.
\]
This proves \eqref{eq:vertex-energy}.
\end{proof}

\begin{lemma}[Deletion monotonicity]\label{lem:deletion}
For every vertex $v$ of a graph of order $n\ge5$,
\begin{equation}\label{eq:deletion}
 \E(G)-\dd(G)\ge \E(G-v)-\dd(G-v).
\end{equation}
\end{lemma}

\begin{proof}
Let $u_1,\ldots,u_{n-1}$ be an orthonormal eigenbasis for the principal
compression $A(G-v)$, extended by zero in coordinate $v$.  Then
\[
 \E(G-v)=\sum_j |u_j^TAu_j|
 \le\sum_j u_j^T|A|u_j=\E(G)-e_v.
\]
Moreover,
\[
 \dd(G)-\dd(G-v)=\frac{2d_v-\dd(G)}{n-1}.
\]
Now apply Lemma~\ref{lem:vertex-energy}.
\end{proof}

\begin{lemma}[Strict deletion]\label{lem:strict-deletion}
Let $G$ be a graph of order $n\ge5$ with at least one edge. Then, for every
vertex $v\in V(G)$,
\begin{equation}\label{eq:strict-deletion}
 \E(G)-\dd(G)>\E(G-v)-\dd(G-v).
\end{equation}
\end{lemma}

\begin{proof}
It is enough to observe that the estimate in
Lemma~\ref{lem:vertex-energy} is strict whenever $G$ has at least one
edge.  Put $d=\dd(G)$ and $d_v=d_G(v)$.  If $2d_v<d$, then the
right-hand side of \eqref{eq:vertex-energy} is negative, whereas
$e_v\ge0$.

If $2d_v=d$, then $d>0$ and hence $d_v>0$.  Since
\[
 d_v=(A^2)_{vv}=\sum_i \lambda_i^2u_i(v)^2>0,
\]
there is some $i$ for which $\lambda_i\ne0$ and $u_i(v)\ne0$.  Therefore
\[
 e_v=\sum_i |\lambda_i|u_i(v)^2>0,
\]
and again the inequality is strict.

Finally, if $2d_v>d$, the last numerical step in the proof of
Lemma~\ref{lem:vertex-energy} is strict, since
\[
 \frac{27}{4}>\frac{4n}{n-1}\qquad(n\ge5).
\]
Tracing the preceding inequalities back therefore gives
\[
 e_v>\frac{2d_v-d}{n-1}.
\]
Thus, in all three cases,
\[
 e_v>\frac{2d_v-d}{n-1}.
\]

Combining this with
\[
 \E(G-v)\le\E(G)-e_v
\]
and
\[
 \dd(G)-\dd(G-v)=\frac{2d_v-d}{n-1}
\]
gives \eqref{eq:strict-deletion}.
\end{proof}

We shall also need a small base case.

\begin{proposition}\label{prop:n5}
Every graph of order five satisfies \eqref{eq:main}.
\end{proposition}

\begin{proof}
Write $d=\dd(G)$ and $r=r(G)$.  Besides \eqref{eq:E-rank}, we have
$\E(G)\ge2\lambda_1(G)\ge2d$.  Thus it is enough to consider
$1<d<r-1$.

If $r=3$, let $\lambda_1,\lambda_2,\lambda_3$ denote the nonzero
eigenvalues.  By Lemma~\ref{lem:pseudodet},
$|\det{}^*A|=|\lambda_1\lambda_2\lambda_3|\ge1$.  The geometric mean of
the three pairwise products $|\lambda_i\lambda_j|$ is therefore
$|\det{}^*A|^{2/3}\ge1$.  Hence, by AM--GM,
\[
 \sum_{i<j}|\lambda_i\lambda_j|\ge\binom32=3.
\]
Consequently,
\[
 \E(G)^2
 =\sum_i\lambda_i^2
 +2\sum_{i<j}|\lambda_i\lambda_j|
 \ge 5d+2\binom32
 =5d+6,
\]
and
\[
 5d+6-(d+2)^2=(2-d)(d+1)>0.
\]
Let $r=4$.  For inertia $(n^+, n^-)=(1,3)$ or $(3,1)$,
\eqref{eq:inertia-alpha} and \eqref{eq:KP} give $\E(G)\ge6>d+3$.
For $(n^+, n^-)=(2,2)$, let $a,b$ be the positive eigenvalues,  and let $c,d'$ be the magnitudes of the negative eigenvalues, and put
\[
  X=ab,\qquad Y=cd'.
\]
Since $\operatorname{tr}A=0$,
\[
  a+b=c+d'=:T,
\]
and hence $\E(G)=2T$.  Moreover,
\[
  5d=\sum_i\lambda_i^2=(a^2+b^2)+(c^2+d'^2)
     =2T^2-2(X+Y).
\]
Therefore
\[
  \E(G)^2=4T^2=10d+4(X+Y).
\]
By Lemma~\ref{lem:pseudodet},
\[
  XY=\left|\prod_{\lambda_i\ne0}\lambda_i\right|\ge1,
\]
so $X+Y\ge2$ by AM--GM.  Consequently,
\[
  \E(G)^2\ge10d+8.
\]
Finally,
\[
  10d+8-(d+3)^2=3-(d-2)^2>0,
\]
because $1<d<3$.  Hence $\E(G)>d+3=r+d-1$.

Finally, if $r=5$, then $\lambda_1(G)\le4<7.11$, so the result follows from
\cite[Theorem 2.1]{AkbariDabirianGhasemi2022}.  The remaining ranks were
already covered by $d\le1$ or $d\ge r-1$.
\end{proof}

\section{A determinant--variance reduction}\label{sec:reduction}

In this section $G$ is connected and nonsingular.  Let
$\rho=\lambda_1(G)$, and suppose, for a contradiction, that
\begin{equation}\label{eq:counter}
 \E(G)<n-1+\dd(G).
\end{equation}
Put
\begin{equation}\label{eq:aST}
 a=\frac{n-1-\dd(G)}2,\qquad S=n-1-a,\qquad \E(G)=2T.
\end{equation}
Throughout this section, the strict inequalities below ultimately arise from the standing counterexample assumption $T<S$.

If the inertia is $(p,q)$, then
\eqref{eq:KP} and \eqref{eq:inertia-alpha} give
\begin{equation}\label{eq:T-pq}
 T\ge n-\alpha(G)\ge\max\{p,q\}.
\end{equation}

Assume first that $p\ge3$ and put $k=p-1\ge2$.  Since
$\dd(G)\le n-1$, we have $a\ge0$.  If $a=0$, then $\dd(G)=n-1$, so
$G=K_n$, contradicting \eqref{eq:counter}.  Hence $a>0$.
Moreover, \eqref{eq:counter} and \eqref{eq:aST} give
$T<S=n-1-a$, whereas \eqref{eq:T-pq} gives $T\ge\max\{p,q\}$.
Since $n=p+q$ and $k=p-1$, we obtain
\[
 q<n-1-a=p+q-1-a,\qquad
 p<n-1-a=p+q-1-a.
\]
Thus $a<k$ and $a<q-1$, and hence
\begin{equation}\label{eq:a-domain}
 0<a<\min\{k,q-1\}.
\end{equation}

By the two ranges from \cite{AkbariDabirianGhasemi2022}, a counterexample
must satisfy
\[
 \rho>7.11,\qquad a<\log n+1.
\]
Since $a<\log n+1$, we have $n>e^{a-1}$.  Moreover,
$\rho\ge\dd(G)=n-1-2a$.  Hence
\begin{equation}\label{eq:h}
 \rho>h(a):=\max\{7.11,e^{a-1}-1-2a\}.
\end{equation}

Write the positive eigenvalues as $\rho,x_1,\ldots,x_k$ and the absolute
values of the negative eigenvalues as $y_1,\ldots,y_q$.  Set
\[
 R=S-\rho.
\]
Because $\sum_i x_i=T-\rho<R$ and $\rho\ge n-1-2a$, we have
\begin{equation}\label{eq:R}
 0<R\le a,\qquad
 \prod_{i=1}^k x_i<\left(\frac Rk\right)^k.
\end{equation}

Define
\[
 U=\sum_{j=1}^q(y_j-1)=T-q,\qquad
 V=\sum_{j=1}^q(y_j-1)^2.
\]
Then $U<k-a$.  Furthermore, \eqref{eq:LTZ} gives
\begin{equation}\label{eq:V-first}
 V=s^-(G)-2T+q>2a-k.
\end{equation}
Since
\[
 \sum_{i=1}^k x_i=T-\rho<S-\rho=R,
\]
and $x_i>0$, we have
\[
 \sum_{i=1}^k x_i^2
 \le
 \left(\sum_{i=1}^k x_i\right)^2
 <R^2.
\]
Hence
\[
 s^+(G)=\rho^2+\sum_{i=1}^k x_i^2
 <\rho^2+R^2.
\]

Using \eqref{eq:squares} and
$T<S$, we obtain
\begin{align}
 V&=n\dd(G)-s^+(G)-2T+q\notag\\
  &>n\dd(G)-(\rho^2+R^2)-2S+q.\notag
\end{align}
Using $\rho+R=S$, $\dd(G)=n-1-2a$, $S=n-1-a$, and
$n-1=k+q$, we obtain
\begin{equation}
 V> n\dd(G)-(\rho^2+R^2)-2S+q
 =2\rho R-a^2-k. \label{eq:V-second}
\end{equation}

We shall use the following elementary analytic lemma.
\begin{lemma}[Logarithmic variance]\label{lem:logvariance}
If $z_j>0$, $U=\sum_j(z_j-1)$, and
$V=\sum_j(z_j-1)^2$, then
\begin{equation}\label{eq:logvariance}
 \sum_j\log z_j\le U-g(V),
 \qquad
 g(V)=\sqrt V-\log(1+\sqrt V).
\end{equation}
\end{lemma}

\begin{proof}
Put
\[
 f(u)=u-\log(1+u),\qquad u>-1.
\]
Then
\[
 \sum_j f(z_j-1)
 =
 U-\sum_j\log z_j.
\]
Thus it is enough to prove
\[
 \sum_j f(z_j-1)\ge g(V).
\]

Assume first that $V>0$.  Since
\[
 V=\sum_j(z_j-1)^2,
\]
we have $|z_j-1|\le\sqrt V$ for every $j$.

We first record three elementary estimates.  For $u>0$, the function
$f(u)/u^2$ is decreasing.  Indeed, if
\[
 N(u):=\frac{u^2}{1+u}-2u+2\log(1+u),
\]
then
\[
 \frac{d}{du}\left(\frac{f(u)}{u^2}\right)
 =\frac{N(u)}{u^3},
 \qquad
 N'(u)=-\frac{u^2}{(1+u)^2}<0,
\]
and $N(0)=0$.  Next, for $-1<u\le0$,
\[
 f(u)\ge\frac{u^2}{2},
\]
because the derivative of
$f(u)-u^2/2$ is $-u^2/(1+u)\le0$ and the difference vanishes at $u=0$.
Finally,
\[
 g(V)\le\frac V2.
\]
To see this, put $s=\sqrt V$ and note that
\[
 \frac{s^2}{2}-s+\log(1+s)
\]
has derivative $s^2/(1+s)\ge0$ and vanishes at $s=0$.

We now show that, for every $j$,
\[
 f(z_j-1)\ge \frac{g(V)}{V}(z_j-1)^2.
\]
If $z_j-1>0$, then
\[
 \frac{f(z_j-1)}{(z_j-1)^2}
 \ge
 \frac{f(\sqrt V)}{V}
 =
 \frac{g(V)}{V}.
\]
If $-1<z_j-1\le0$, then
\[
 f(z_j-1)
 \ge
 \frac{(z_j-1)^2}{2}
 \ge
 \frac{g(V)}{V}(z_j-1)^2.
\]
Summing over $j$ gives
\[
 \sum_j f(z_j-1)\ge g(V),
\]
which is equivalent to \eqref{eq:logvariance}.

If $V=0$, then $z_j=1$ for every $j$, and
\eqref{eq:logvariance} is immediate.
\end{proof}

We now return to the determinant.
Since $G$ is nonsingular,
\[
 |\det A|
 =\rho\prod_{i=1}^k x_i\prod_{j=1}^q y_j.
\]
By \eqref{eq:R},
\[
 \sum_{i=1}^k\log x_i
 <k\log(R/k).
\]
Moreover, Lemma~\ref{lem:logvariance}, applied with $z_j=y_j$, gives
\[
 \sum_{j=1}^q\log y_j\le U-g(V).
\]
Since
\[
 U=T-q<S-q=k-a,
\]
we obtain
\begin{equation}\label{eq:det-first}
 \log|\det A|
 <\log\rho+k\log(R/k)+k-a-g(V).
\end{equation}

Put
\begin{equation}\label{eq:wZtau}
 \begin{aligned}
 w&=\max\{0,2a-k,2R\rho-a^2-k\},\\
 Z&=w+a^2+k,
 &\tau&=\min\left\{a,\frac{Z}{2h(a)}\right\}.
 \end{aligned}
\end{equation}
Then $2R\rho\le Z$, and \eqref{eq:R}, \eqref{eq:h} imply $R\le\tau$.
Consequently
\begin{equation}\label{eq:rhork}
  \rho R^k\le\frac Z2\tau^{k-1}.
\end{equation}

By \eqref{eq:V-first} and \eqref{eq:V-second},
\[
 V>2a-k,\qquad V>2R\rho-a^2-k.
\]
Moreover, $V>0$: otherwise $y_j=1$ for every $j$, so
$s^-(G)=q\le n-3$, contradicting \eqref{eq:LTZ}.  Hence
\[
 V>w.
\]
Since $g$ is increasing,
\[
 -g(V)<-g(w).
\]

Combining \eqref{eq:det-first}, \eqref{eq:rhork}, and $V>w$, and using
the monotonicity of $g$, we obtain
\begin{equation}\label{eq:det-F}
 \log|\det A|<F(k,a,w).
\end{equation}
Here
\begin{equation}\label{eq:F}
 F(k,a,w)=
 \log\frac Z2+(k-1)\log\tau-k\log k+k-a-g(w).
\end{equation}
The remaining domain is simply
\begin{equation}\label{eq:F-domain}
 k\in\mathbb Z,\qquad k\ge2,\qquad 0<a<k,\qquad
 w\ge\max\{0,2a-k\}.
\end{equation}

\section{The scalar inequality}\label{sec:scalar}

Write
\[
 s=\sqrt w,\qquad
 H:=\frac{711}{100}=7.11,\qquad
 Z=s^2+a^2+k.
\]
Thus
\[
 h(a)=\max\{H,e^{a-1}-1-2a\}.
\]
Let $a_0$ be the unique solution, on the increasing branch of
$e^{a-1}-1-2a$, of
\[
 e^{a-1}-1-2a=H.
\]
Since the derivative of $e^{a-1}-1-2a$ is $e^{a-1}-2$, this function is
strictly increasing for $a>1+\log2$.  The exponential series, with the
standard geometric estimate for its tail,
\[
 0<e^x-\sum_{j=0}^N\frac{x^j}{j!}
 <
 \frac{x^{N+1}}{(N+1)!}
 \frac{1}{1-x/(N+2)}
 \qquad (N+2>x),
\]
gives, by taking $N=3$ for the upper bound and the ninth partial sum for
the lower bound,
\[
 e^{27/10}<\frac{1551}{100},
 \qquad
 e^{11/4}>\frac{1561}{100}.
\]
Consequently,
\[
 e^{27/10}-1-\frac{74}{10}<H
 <e^{11/4}-1-\frac{15}{2},
\]
and hence
\begin{equation}\label{eq:a0}
 \frac{37}{10}<a_0<\frac{15}{4}.
\end{equation}
Thus
\[
 h(a)=H\quad(0<a\le a_0),
 \qquad
 h(a)=e^{a-1}-1-2a\quad(a\ge a_0).
\]

We shall also use
\begin{equation}\label{eq:log-lower}
 \log x>\frac{2(x-1)}{x+1}\qquad(x>1).
\end{equation}
Indeed, for
\[
 \psi(x)=\log x-\frac{2(x-1)}{x+1},
\]
we have
\[
 \psi'(x)=\frac{(x-1)^2}{x(x+1)^2}>0
 \qquad(x>1),
\]
while $\psi(1)=0$.  In particular,
\[
 \log2>\frac23,
 \qquad
 \log H>
 \frac{2(H-1)}{H+1}
 =
 \frac{1222}{811}
 >
 \frac32.
\]
For the complementary upper bounds, the exponential series gives
\[
 e^{7/10}>
 1+\frac7{10}+\frac{49}{200}+\frac{343}{6000}>2
\]
and
\[
 e^2>
 1+2+2+\frac43+\frac23+\frac4{15}
 =\frac{109}{15}>H.
\]
Thus
\[
 \frac23<\log2<\frac7{10},
 \qquad
 \frac32<\log H<2.
\]
Since proving the desired inequality on a larger domain is sufficient, for
each fixed $k\ge2$ we enlarge \eqref{eq:F-domain} and allow
\[
 0<a<k,\qquad s\ge0.
\]
For fixed $k$, the function $F$ is piecewise smooth in $(a,s)$.  The next
lemma shows that possible interior maxima can be moved to the boundary of
the two branches, reducing the proof to two boundary families.

\begin{lemma}[Reduction to boundary cases]\label{lem:s-reduction}
Fix an integer $k\ge2$.  To prove
\[
 F(k,a,s^2)<0
 \qquad(0<a<k,\ s\ge0),
\]
it is enough to consider the cases
\[
 s=0
 \qquad\text{and}\qquad
 Z=2ah(a).
\]
\end{lemma}

\begin{proof}
On the branch $Z\ge2ah(a)$ we have $\tau=a$; denote the corresponding
function by $F_A$.  On the branch $Z\le2ah(a)$ we have
$\tau=Z/(2h(a))$; denote it by $F_B$.  Since
\[
 g(s^2)=s-\log(1+s),
 \qquad
 \frac{\partial Z}{\partial s}=2s,
\]
we obtain
\begin{align}
 \frac{\partial F_A}{\partial s}
 &=\frac{2s}{Z}-1+\frac1{1+s}\notag\\
 &=\frac{s\{3-a^2-k-(s-1)^2\}}{Z(1+s)},
 \label{eq:FA-s}\\
 \frac{\partial F_B}{\partial s}
 &=\frac{2ks}{Z}-1+\frac1{1+s}\notag\\
 &=\frac{s\{k^2+k-a^2-(s-k)^2\}}{Z(1+s)}.
 \label{eq:FB-s}
\end{align}

For fixed $a$, the only remaining candidates for an interior maximum in
the $s$-variable are the upper critical points described below.  We show
that the values along each such critical curve are bounded above by a
value attained when the curve meets the branch boundary.

Apart from $s=0$ and the branch boundary, an interior critical maximum of
$F_A$ can therefore occur only at
\[
 s=1+\sqrt{3-a^2-k}.
\]
This is possible only for $k=2$, in which case
\[
 s=1+\sqrt{1-a^2},
 \qquad 0<a\le1.
\]
Along this critical curve, $\frac{\partial F_A}{\partial s}=0$, and hence, by the chain rule,
\[
 \frac{d}{da}F_A(a,s(a))=\frac{\partial F_A}{\partial a}.
\]
The critical-point equation gives $Z=2(1+s)$, so
\[
 \frac{d}{da}F_A(a,s(a))
 =
 \frac1a-1+\frac{a}{1+s}>0.
\]
Moreover,
\[
 s'(a)=-\frac{a}{\sqrt{1-a^2}}<0.
\]
Here $a\le1<a_0$, so $h(a)=H$.  Along the critical curve define
\[
 D(a):=Z-2aH=2(1+s)-2aH.
\]
Then
\[
 D'(a)
 =
 -\frac{2a}{\sqrt{1-a^2}}-2H<0,
\]
while
\[
 \lim_{a\to0^+}D(a)=6>0,
 \qquad
 D(1)=4-2H<0.
\]
Thus the feasible part of this critical curve in branch $A$ ends at a
unique branch-boundary point.  Since $F_A$ increases along the critical
curve, its largest feasible value is attained there.  The lower critical
root, when positive, is a local minimum.

For branch $B$, the only possible interior critical maximum is
\[
 s=k+\sqrt{k^2+k-a^2}.
\]
At such a point, \eqref{eq:FB-s} gives
\[
 Z=2k(1+s).
\]
Along this critical curve,
\[
 \frac{d}{da}F_B(a,s(a))
 =
 \frac{a}{1+s}-1-(k-1)\frac{h'(a)}{h(a)},
\]
where the last term is absent on the constant branch.  This derivative is
strictly negative, since $s>k>a$.  The branch-$B$ condition is
\[
 k(1+s)\le ah(a).
\]
If
\[
 D(a):=ah(a)-k(1+s),
\]
then
\[
 D'(a)
 =
 h(a)+ah'(a)+\frac{ka}{\sqrt{k^2+k-a^2}}>0
\]
on either branch of $h$.  Thus the feasible portion of this critical curve in branch $B$ begins
at a branch-boundary point; since $F_B$ decreases along the curve, its
largest feasible value is attained there.  Finally,
$F_A(k,a,s^2)\to-\infty$ as $s\to\infty$.  This proves the reduction.
\end{proof}

\begin{lemma}\label{lem:s-zero}
For every integer $k\ge2$ and every $0<a<k$,
\[
 F(k,a,0)<0.
\]
\end{lemma}

\begin{proof}
We first note that
\[
 h(a)>a,\qquad 2h(a)-a>10
 \qquad(a>0).
\]
Indeed, these inequalities follow from $h(a)=H$ and
$a\le a_0<15/4$ when $a\le a_0$.  On the exponential branch,
\[
 (h(a)-a)'=e^{a-1}-3>0,
 \qquad
 (2h(a)-a)'=2e^{a-1}-5>0,
\]
and the inequalities already hold at $a=a_0$.

Suppose first that $s=0$ lies on branch $A$.  Put $r=k/a$.  The branch
condition gives
\[
 r\ge2h(a)-a>10.
\]
Also $h(a)>a$, so $r>a$.  A direct simplification gives
\[
 F_A(k,a,0)
 =
 \log\frac{a+r}{2}+a(r-1-r\log r).
\]
Since $r>a$ and, by \eqref{eq:log-lower},
\[
 \log r>\log10=\log2+\log5>\frac23+\frac43=2,
\]
we obtain
\[
 \begin{aligned}
 F_A(k,a,0)
 &<\log r+k(1-\log r)-a\\
 &\le 2-a-\log r<0.
 \end{aligned}
\]

Now suppose that $s=0$ lies on branch $B$ and first assume $a\le a_0$, so
$h(a)=H$.  Then
\[
 F_B(k,a,0)
 =
 k\log\frac{a^2+k}{2k}-(k-1)\log H+k-a.
\]
For fixed $k$,
\[
 \frac{\partial F_B}{\partial a}
 =
 \frac{2ka}{a^2+k}-1.
\]
Its only critical point in $(0,k)$ is a local minimum.  Hence the maximum
on each feasible interval occurs at an endpoint.  The branch-boundary
endpoints will be handled in Lemma~\ref{lem:constant-boundary}.  At the
other endpoint we obtain, for $k=2$ and $k=3$,
\[
 2\log(3/2)-\log H<0,
 \qquad
 3\log2-2\log H<0.
\]
For $k\ge4$, the remaining endpoint is $a=a_0$.  Put
\[
 x=1+\frac{a_0^2}{k}.
\]
Treating $k$ temporarily as a continuous variable,
\[
 \frac{\partial F_B(k,a_0,0)}{\partial k}
 =
 \log\frac{x}{2H}+\frac1x.
\]
Since $a_0<15/4$ and $k\ge4$,
\[
 x\le\frac{289}{64}<\frac{2H}{3},
\]
and therefore
\[
 \frac{\partial F_B(k,a_0,0)}{\partial k}
 <\log(1/3)+1<0.
\]
Thus $k=4$ is the worst case.  Using \eqref{eq:a0},
$\log H>3/2$, and $289/128<e$, we get
\[
 F_B(4,a_0,0)
 <
 4-\frac92+\frac3{10}
 =
 -\frac15.
\]

Finally, on the exponential branch,
\[
 h=e^{a-1}-1-2a,
 \qquad
 \frac{h'}h>1,
\]
and hence
\[
 \frac{\partial F_B(k,a,0)}{\partial a}
 =
 \frac{2ka}{a^2+k}-1-(k-1)\frac{h'}h
 <
 k\left(\frac{2a}{a^2+k}-1\right)<0.
\]
Thus its maximum occurs either at $a=a_0$, already covered above, or at a
branch boundary.  This completes the proof.
\end{proof}

\begin{lemma}\label{lem:constant-boundary}
Suppose $0<a\le a_0$, so $h(a)=H$, and suppose
\(
 Z=2Ha.
\)
Then
\[
 F(k,a,s^2)<0.
\]
\end{lemma}

\begin{proof}
Put
\[
 \phi(x)=x-1-\log x,\qquad 0<x<1.
\]
At the branch boundary,
\[
 s^2=2Ha-a^2-k
\]
and
\begin{equation}\label{eq:constant-boundary-F}
 F
 =
 \log H-k\phi(a/k)-g(s^2).
\end{equation}
Thus it is enough to prove
\[
 P:=k\phi(a/k)+g(s^2)>2,
\]
because $\log H<2$.

We first reduce to $k=2$.  Fix $s$.  Along the branch boundary,
\[
 a=H-\sqrt{H^2-k-s^2}.
\]
Treating $k$ as a continuous variable and writing $r=k/a>1$, we have
\[
 \frac{da}{dk}=\frac{1}{2(H-a)}
\]
and
\[
 \frac{d}{dk}\bigl(k\phi(a/k)\bigr)
 =
 \log r-\frac{r-1}{2(H-a)}.
\]
Moreover,
\[
 r=2H-a-\frac{s^2}{a}\le2H-a.
\]
Since $a\le a_0<15/4$,
\[
 r+1<4(H-a).
\]
By \eqref{eq:log-lower},
\[
 \log r>
 \frac{2(r-1)}{r+1}
 >
 \frac{r-1}{2(H-a)}.
\]
Hence $P$ is strictly increasing with $k$ along the branch boundary.

If $s\ge4$, then
\[
 P\ge g(16)=4-\log5>2.
\]
Suppose therefore that $s<4$.  The branch-boundary point with $k=2$ is
feasible, since
\[
 s^2<16<4H-6
\]
implies that its smaller root satisfies $0<a<2$.  Hence it remains only to
prove
\begin{equation}\label{eq:P2}
 2\phi(a/2)+g(s^2)>2,
 \qquad
 s^2=2Ha-a^2-2\ge0,
 \qquad
 0<a<2.
\end{equation}

If $0<a\le1/4$, then $\phi$ is decreasing on $(0,1)$ and
\[
 2\phi(a/2)
 \ge
 2\phi(1/8)
 =
 -\frac74+6\log2
 >
 \frac94.
\]

If $1/4\le a\le3/8$, then
\[
 2\phi(a/2)
 \ge
 2\phi(3/16)
 =
 -\frac{13}{8}+2\log\frac{16}{3}.
\]
By \eqref{eq:log-lower},
\[
 \log\frac{16}{3}
 =
 2\log2+\log\frac43
 >
 \frac43+\frac27
 =
 \frac{34}{21}.
\]
Also $s^2$ is increasing in $a$, and at $a=1/4$,
\[
 s^2=\frac{597}{400}>
 \left(\frac65\right)^2.
\]
The elementary estimate
\[
 g(s^2)\ge\frac{3s^2}{2(2s+3)}
\]
follows by considering
\[
 q(s)=s-\log(1+s)-\frac{3s^2}{2(2s+3)}.
\]
Indeed,
\[
 q'(s)=\frac{s^3}{(1+s)(2s+3)^2}\ge0,
 \qquad q(0)=0.
\]
Moreover, $3s^2/[2(2s+3)]$ is strictly increasing for $s>0$.
Therefore $s>6/5$ gives $g(s^2)>2/5$.  Hence
\[
 2\phi(a/2)+g(s^2)
 >
 -\frac{13}{8}+\frac{68}{21}+\frac25
 =
 \frac{1691}{840}>2.
\]

Suppose next that
\[
 \frac38\le a\le\frac65.
\]
Since the function $a\mapsto2\phi(a/2)$ is convex, its tangent at
$a=3/4$ gives
\[
 2\phi(a/2)
 \ge
 2\log\frac83-\frac{5a}{3}.
\]
By concavity of $\log$, its tangent at $4$ gives
\[
 g(s^2)=s-\log(1+s)
 \ge
 \frac{3s}{4}+\frac34-\log4.
\]
Furthermore,
\[
 2Ha-a^2-2-
 \left(\frac{20a}{9}+\frac{19}{20}\right)^2
\]
is a concave quadratic in $a$ whose values at $3/8$ and $6/5$ are,
respectively,
\[
 \frac{167}{14400}
 \qquad\text{and}\qquad
 \frac{9787}{18000}.
\]
Thus
\[
 s>\frac{20a}{9}+\frac{19}{20}.
\]
Consequently,
\[
 \begin{aligned}
 2\phi(a/2)+g(s^2)
 &>
 \log\frac{16}{9}+\frac{117}{80}\\
 &>
 \frac{14}{25}+\frac{117}{80}
 =
 \frac{809}{400}>2,
 \end{aligned}
\]
where the second inequality follows from \eqref{eq:log-lower}.

Finally, if $6/5\le a<2$, then $s^2$ is increasing and
\[
 s^2>
 14\cdot\frac65-\left(\frac65\right)^2-2
 =
 \frac{334}{25}
 >
 \left(\frac{18}{5}\right)^2.
\]
Hence $s>18/5$.  By concavity of $\log$ at $4$,
\[
 \log\frac{23}{5}
 <
 \log4+\frac3{20}
 <
 \frac{31}{20},
\]
and therefore
\[
 g(s^2)>
 \frac{18}{5}-\frac{31}{20}
 =
 \frac{41}{20}>2.
\]
This proves \eqref{eq:P2}, and hence \eqref{eq:constant-boundary-F} is
negative.
\end{proof}

\begin{lemma}\label{lem:exponential-boundary}
Suppose $a\ge a_0$, $h(a)=e^{a-1}-1-2a$, and $Z=2ah(a)$. Then
\[
 F(k,a,s^2)<0.
\]
\end{lemma}

\begin{proof}
Put
\[
 r=\frac{k}{a}>1,
 \qquad
 W=s^2=a(2h-a-r).
\]
At the branch boundary,
\begin{equation}\label{eq:exp-boundary}
 F
 =
 \log h+a(r-1-r\log r)-g(W).
\end{equation}
For fixed $a$,
\[
 \frac{\partial F}{\partial r}
 =
 a\left(-\log r+\frac1{2(1+\sqrt W)}\right).
\]
Hence
\[
 \frac{\partial F}{\partial r}<0
 \qquad
 \text{for }r\ge\frac53,
\]
since $\log(5/3)>1/2$.

Temporarily allowing $r$ (equivalently, $k=ar$) to vary continuously,
if a feasible point has $r>5/3$, then replacing $r$ by $5/3$ preserves
feasibility, because $W=a(2h-a-r)$ increases as $r$ decreases.  Since $F$
is decreasing for $r\ge5/3$, it is therefore enough to consider
$1<r\le5/3$.  Since
\[
 r-1-r\log r\le0
\]
and
\[
 W\ge W_*:=a\left(2h-a-\frac53\right),
\]
we obtain
\[
 F\le\log h-g(W_*).
\]
Moreover,
\[
 \log h<a-1,
\]
because $h<e^{a-1}$.  Thus it remains to prove
\[
 g(W_*)>a-1.
\]

First suppose $a_0\le a\le4$.  Since $h\ge H$ and the function
\[
 a\longmapsto a\left(2H-a-\frac53\right)
\]
is increasing on $[a_0,4]$, \eqref{eq:a0} gives
\[
 W_*
 \ge
 a\left(2H-a-\frac53\right)
 >
 \frac{12284}{375}
 >
 \left(\frac{57}{10}\right)^2.
\]
Hence
\[
 g(W_*)
 >
 \frac{57}{10}-\log\frac{67}{10}
 >
 \frac{37}{10}
 >
 a-1,
\]
where $\log(67/10)<2$ follows from $67/10<e^2$.

Now suppose $a\ge4$.  We claim that
\[
 2h>5a+\frac53.
\]
At $a=4$ this follows from $e^3>20$, and the derivative of the difference
between the two sides is
\[
 2h+4a-7>0.
\]
Therefore
\[
 W_*>4a^2,
\qquad
 \sqrt{W_*}>2a\ge8.
\]
For $s\ge8$ one has $\log(1+s)<s/2$, so
\[
 g(W_*)>\frac{\sqrt{W_*}}2>a>a-1.
\]
This completes the proof.
\end{proof}

\begin{lemma}\label{lem:scalar}
On the domain \eqref{eq:F-domain}, the function in \eqref{eq:F} satisfies
\begin{equation}\label{eq:F-negative}
 F(k,a,w)<0.
\end{equation}
\end{lemma}

\begin{proof}
By Lemma~\ref{lem:s-reduction}, it is enough to consider $s=0$ and the
branch boundary $Z=2ah(a)$.  The case $s=0$ is
Lemma~\ref{lem:s-zero}.  At the branch boundary, either $a\le a_0$ and
Lemma~\ref{lem:constant-boundary} applies, or $a\ge a_0$ and
Lemma~\ref{lem:exponential-boundary} applies.
\end{proof}

\begin{proposition}\label{prop:pge3}
No connected nonsingular counterexample has $p\ge3$.
\end{proposition}

\begin{proof}
For such a counterexample, \eqref{eq:det-F} and
Lemma~\ref{lem:scalar} give $\log|\det A|<0$.  This contradicts
$|\det A|\ge1$.
\end{proof}

\section{Two positive eigenvalues}\label{sec:p2}

It remains to handle the case excluded by $k=p-1\ge2$.

\begin{proposition}\label{prop:p2}
Let $G$ be connected and nonsingular of order $n\ge5$.  If $p(G)=2$, then
\[
 \E(G)\ge n-1+\dd(G).
\]
\end{proposition}

\begin{proof}
The case $n=5$ follows from Proposition~\ref{prop:n5}, so assume $n\ge6$.
Let
\[
 B=A(\overline G),
 \qquad
 \overline m=|E(\overline G)|.
\]
Since $q=n-2$, \eqref{eq:inertia-alpha} gives $\alpha(G)\le2$; hence
$\overline G$ is triangle-free.  If $\overline m\ge n$, then
\eqref{eq:KP} gives
\[
 \E(G)\ge2n-4
 \ge2n-2-\frac{2\overline m}{n}
 =n-1+\dd(G).
\]
We may therefore assume
\begin{equation}\label{eq:mbar-small}
 \overline m<n.
\end{equation}

Suppose first that $\overline G$ is nonbipartite.  A shortest odd cycle in
$\overline G$ is induced and has length $\ell\ge5$.  Besides the eigenvalue
$\ell-3$, the graph $\overline{C_\ell}$ has eigenvalues
\[
 -1-2\cos\frac{2\pi j}{\ell},
 \qquad j=1,\ldots,\ell-1.
\]
For
\[
 j=\frac{\ell-1}{2},\qquad j=\frac{\ell+1}{2},
\]
both values equal
\[
 -1+2\cos\frac{\pi}{\ell}>0.
\]
Thus the induced subgraph $\overline{C_\ell}$ of $G$ has at least three
positive eigenvalues, contradicting interlacing and $p(G)=2$.  Hence
$\overline G$ is bipartite.

Suppose that a bipartite component $C$ of $\overline G$, with bipartition
classes of sizes $u$ and $v$, contains a cycle.  Let $\mathbf{x}$ take the
value $1/u$ on the first class, $-1/v$ on the second, and $0$ elsewhere.
Then $\mathbf{x}\perp\mathbf1$, and, since $A(G)=J-I-B$,
\[
 \mathbf{x}^TA(G)\mathbf{x}
 =
 -\|\mathbf{x}\|^2-\mathbf{x}^TB\mathbf{x}.
\]
Moreover,
\[
 \|\mathbf{x}\|^2
 =
 \frac1u+\frac1v
 =
 \frac{u+v}{uv},
\]
while
\[
 \mathbf{x}^TB\mathbf{x}
 =
 2\sum_{ij\in E(C)}x_ix_j
 =
 -\frac{2|E(C)|}{uv}.
\]
Therefore
\begin{equation}\label{eq:cycle-rayleigh}
 \frac{\mathbf{x}^TA(G)\mathbf{x}}{\|\mathbf{x}\|^2}
 =
 \frac{2|E(C)|}{u+v}-1
 \ge1,
\end{equation}
because a connected graph containing a cycle satisfies
$|E(C)|\ge|V(C)|=u+v$.

Put
\[
 \mathbf e=\frac{\mathbf1}{\sqrt n},
 \qquad
 \widehat{\mathbf{x}}=\frac{\mathbf{x}}{\|\mathbf{x}\|}.
\]
The vectors $\mathbf e$ and $\widehat{\mathbf{x}}$ are orthonormal, so
Ky Fan's principle gives
\[
 \lambda_1(G)+\lambda_2(G)
 \ge
 \mathbf e^TA(G)\mathbf e+
 \widehat{\mathbf{x}}^TA(G)\widehat{\mathbf{x}}
 \ge
 \dd(G)+1.
\]
Since $p(G)=2$,
\[
 \frac{\E(G)}2=\lambda_1(G)+\lambda_2(G)\ge\dd(G)+1.
\]
Using
\[
 \dd(G)=n-1-\frac{2\overline m}{n}
\]
and \eqref{eq:mbar-small}, we obtain
\[
 \dd(G)+1
 =
 n-\frac{2\overline m}{n}
 >
 n-1-\frac{\overline m}{n}
 =
 \frac{n-1+\dd(G)}2.
\]
Hence the desired inequality is strict in this case.  Therefore every
component of $\overline G$ is acyclic, and $\overline G$ is a forest.

A $K_2$-component of $\overline G$ produces two identical rows of $A(G)$,
contradicting the nonsingularity of $G$.  If $\overline G$ had two
components of order at least three, each would contain an induced $P_3$,
so $G$ would contain an induced $\overline{2P_3}$.  Its characteristic
polynomial is
\[
 (x+1)^2(x^2-4x+1)(x^2+2x-1),
\]
and hence it has three positive eigenvalues, again contradicting
interlacing.  Thus
\[
 \overline G=T\cup tK_1
\]
for a tree $T$ of order at least three.  If $\operatorname{diam}T\ge4$,
then $G$ contains an induced $\overline{P_5}$.  Its characteristic
polynomial is
\[
 x(x+2)(x^3-2x^2-2x+2),
\]
so its third largest eigenvalue is $0$.  Interlacing would then give
$\lambda_3(G)\ge0$, contradicting $p(G)=2$ and nonsingularity.  Hence
$T$ is a star or a double star.

We claim that in either case there exists a nonzero vector
$\mathbf{x}\perp\mathbf1$ such that
\begin{equation}\label{eq:B-witness}
 \frac{\mathbf{x}^TB\mathbf{x}}{\|\mathbf{x}\|^2}
 \le
 -1-\frac{\overline m}{n}.
\end{equation}

First suppose that $T$ is a double star with leaf counts $a,b\ge1$.  Put
\[
 S=a+b,\qquad P=ab,
\]
and let $t$ be the number of isolated vertices of $\overline G$.  Then
\[
 \overline m=S+1,
 \qquad
 n=S+t+2.
\]
Give each vertex of $T$ its degree, with opposite signs on the two
bipartition classes, and give the isolated vertices value $0$.  The
resulting vector $\mathbf{x}$ is orthogonal to $\mathbf1$, and a direct
calculation gives
\[
 \|\mathbf{x}\|^2=(S+1)(S+2)-2P,
\]
and
\[
 -\mathbf{x}^TB\mathbf{x}
 =
 2\bigl((S+1)^2-P\bigr).
\]
After clearing denominators, \eqref{eq:B-witness} is equivalent to
\[
 (S+1)(2P+St-S-2)\ge0.
\]
Since
\[
 P=ab\ge a+b-1=S-1,
\]
this follows immediately.  Indeed, if $t=0$, then $n=S+2\ge6$, so
$S\ge4$, and
\[
 2P-S-2\ge S-4\ge0.
\]
If $t\ge1$, then $S\ge2$, and
\[
 2P+St-S-2
 \ge
 S(t+1)-4
 \ge0.
\]

Now suppose that $T=K_{1,a}$ with $a\ge2$.  We have $t\ge1$, since
otherwise the center of $T$ would be an isolated vertex of $G$.  Here
\[
 \overline m=a,
 \qquad
 n=a+t+1.
\]
Put
\[
 c=1+\frac an=1+\frac{\overline m}{n}.
\]
For arbitrary real numbers $X,Y$, assign the value $X$ to the center,
the value $Y$ to every leaf, and the value
\[
 Z=-\frac{X+aY}{t}
\]
to every isolated vertex of $\overline G$.  This defines a family of
vectors $\mathbf{x}=\mathbf{x}(X,Y)$ satisfying
$\mathbf{x}\perp\mathbf1$.  Substitution gives
\[
 \mathbf{x}^T(B+cI)\mathbf{x}
 =
 \begin{pmatrix}X&Y\end{pmatrix}
 M
 \begin{pmatrix}X\\Y\end{pmatrix},
\]
where
\[
 M=
 \begin{pmatrix}
 c(1+1/t)&a(1+c/t)\\
 a(1+c/t)&ca(1+a/t)
 \end{pmatrix}.
\]
Its determinant is
\[
 \det M
 =
 \frac{a}{t^2}
 \left[c^2(t+1)(t+a)-a(t+c)^2\right].
\]
Furthermore,
\[
 a(t+c)^2-c^2(t+1)(t+a)
 =
 \frac tn\bigl[(a+t)(t(a-1)-2)-1\bigr].
\]
The last quantity is positive.  If $t=1$, then $a\ge4$ and
\[
 (a+1)(a-3)-1>0.
\]
If $a=2$, then $t\ge3$ and
\[
 (t+2)(t-2)-1=t^2-5>0.
\]
In all remaining cases $a\ge3$ and $t\ge2$, so
$t(a-1)-2\ge2$ and the claim is immediate.  Hence $\det M<0$.
Since $M$ is symmetric, there is a nonzero pair $(X,Y)$ for which
\[
 \mathbf{x}^T(B+cI)\mathbf{x}<0.
\]
Thus the corresponding nonzero vector $\mathbf{x}\perp\mathbf1$ satisfies
\[
 \frac{\mathbf{x}^TB\mathbf{x}}{\|\mathbf{x}\|^2}
 <-c
 =
 -1-\frac{\overline m}{n},
\]
which proves \eqref{eq:B-witness}.

Finally, let $\mathbf{x}$ be a unit vector satisfying
\eqref{eq:B-witness}, and put $\mathbf e=\mathbf1/\sqrt n$.  Since
$A(G)=J-I-B$ and $\mathbf{x}\perp\mathbf1$,
\[
 \mathbf e^TA(G)\mathbf e
 =
 n-1-\frac{2\overline m}{n},
\]
while
\[
 \mathbf{x}^TA(G)\mathbf{x}
 =
 -1-\mathbf{x}^TB\mathbf{x}
 \ge
 \frac{\overline m}{n}.
\]
Ky Fan's principle therefore gives
\[
 \lambda_1(G)+\lambda_2(G)
 \ge
 n-1-\frac{\overline m}{n}.
\]
Since $p(G)=2$,
\[
 \E(G)
 =
 2(\lambda_1+\lambda_2)
 \ge
 2n-2-\frac{2\overline m}{n}
 =
 n-1+\dd(G).
\]
\end{proof}

\section{Completion of the proof}\label{sec:completion}

We first make the reduction from connected to arbitrary nonsingular graphs
self-contained.

\begin{lemma}[Disconnected nonsingular graphs]\label{lem:disconnected}
Let $G$ be a disconnected nonsingular graph of order $n\ge5$, with
connected components
\[
 G=G_1\cup\cdots\cup G_c.
\]
Suppose that
\[
 \E(H)\ge |V(H)|-1+\dd(H)
\]
holds for every connected nonsingular graph $H$ of order at least five.
Then
\[
 \E(G)\ge n-1+\dd(G).
\]

Moreover, if equality in the above bound for a connected nonsingular graph
$H$ of order at least five occurs only when $H$ is complete, then equality
for $G$ occurs if and only if
\[
 G\cong\frac n2K_2.
\]
\end{lemma}

\begin{proof}
Write
\[
 n_i=|V(G_i)|,\qquad d_i=\dd(G_i),
\]
and put
\[
 \varepsilon_i
 :=
 \E(G_i)-(n_i-1+d_i).
\]
Since $G$ is nonsingular, every component $G_i$ is nonsingular and has at
least two vertices.  Also,
\[
 \E(G)=\sum_{i=1}^c\E(G_i),
 \qquad
 \dd(G)=\frac1n\sum_{i=1}^c n_i d_i.
\]
Consequently,
\begin{equation}\label{eq:component-gap}
 \E(G)-(n-1+\dd(G))
 =
 \sum_{i=1}^c\varepsilon_i+
 \sum_{i=1}^c
 \left(1-\frac{n_i}{n}\right)(d_i-1).
\end{equation}

Since each $G_i$ is connected and has at least two vertices, $d_i\ge1$.
Thus every term in the second sum of \eqref{eq:component-gap} is
nonnegative.

If $n_i\ge5$, the hypothesis gives $\varepsilon_i\ge0$.  For orders two
and three, the only connected nonsingular graphs are $K_2$ and $K_3$,
respectively, and in both cases $\varepsilon_i=0$.  At order four, the
connected nonsingular graphs are
\[
 K_4,\qquad P_4,\qquad\text{and the paw}.
\]
Again $\varepsilon(K_4)=0$.  Thus only $P_4$ and the paw require separate
consideration.

For $P_4$,
\[
 \dd(P_4)=\frac32,
 \qquad
 \E(P_4)=2\sqrt5,
\]
so
\[
 \varepsilon(P_4)=2\sqrt5-\frac92.
\]
If $P_4$ is a component of the disconnected graph $G$, then $n\ge6$.
Its total contribution to the right-hand side of
\eqref{eq:component-gap} is therefore at least
\[
 2\sqrt5-\frac92+
 \left(1-\frac46\right)\left(\frac32-1\right)
 =
 2\sqrt5-\frac{13}{3}>0.
\]

For the paw, $\dd(\mathrm{paw})=2$.  Its inertia is $(2,2)$.  Let $X$ and
$Y$ denote the products of the positive and negative eigenvalue
magnitudes, respectively.  Since
\[
 XY=|\det A|=1
\]
and the paw has four edges,
\[
 \left(\frac{\E(\mathrm{paw})}{2}\right)^2
 =4+X+Y\ge6.
\]
Hence
\[
 \varepsilon(\mathrm{paw})
 =
 \E(\mathrm{paw})-5
 \ge2\sqrt6-5.
\]
Again $n\ge6$, so its total contribution to \eqref{eq:component-gap} is at
least
\[
 2\sqrt6-5+
 \left(1-\frac46\right)(2-1)
 =
 2\sqrt6-\frac{14}{3}>0.
\]

Thus every component makes a nonnegative contribution to
\eqref{eq:component-gap}, and therefore
\[
 \E(G)\ge n-1+\dd(G).
\]

Now assume in addition that equality for a connected nonsingular graph of
order at least five occurs only for a complete graph, and suppose that
equality holds for $G$.  The preceding estimates show that neither $P_4$
nor the paw can occur as a component.  Notice also that a complete
component $K_{n_i}$ with $n_i\ge3$ has $\varepsilon_i=0$ but
$d_i=n_i-1>1$, so it contributes strictly positively through the second
sum in \eqref{eq:component-gap}.  More generally, every remaining term in
\eqref{eq:component-gap} must vanish.  In particular,
\[
 \left(1-\frac{n_i}{n}\right)(d_i-1)=0
\]
for every $i$.  Since $G$ is disconnected, $n_i<n$, and hence $d_i=1$
for every component.  A connected simple graph of average degree one is
necessarily $K_2$.  Thus every component is $K_2$, so
\[
 G\cong\frac n2K_2.
\]
Conversely, a perfect matching satisfies
\[
 \E(G)=n,\qquad \dd(G)=1,
\]
and therefore attains equality.
\end{proof}

\begin{proof}[Proof of Theorem~\ref{thm:main}]
First suppose that $G$ is nonsingular.  By Lemma~\ref{lem:disconnected},
it is enough to consider connected graphs.  If $p=1$, then
\eqref{eq:inertia-alpha} gives $\alpha(G)=1$, so $G=K_n$.  The cases
$p=2$ and $p\ge3$ follow from Propositions~\ref{prop:p2} and
\ref{prop:pge3}.

Now suppose that $G$ is singular, and put $r=r(G)<n$.  If $r=0$, then
$G$ is empty and \eqref{eq:main} is immediate.  Assume $r>0$.  By
Lemma~\ref{lem:pseudodet}, $G$ has an induced nonsingular subgraph $H$ of
order $r$.

If $r\ge5$, delete the vertices outside $H$ one at a time.  Since $H$
remains an induced subgraph throughout the deletion process, every
intermediate graph contains the nonsingular principal submatrix $A(H)$
and therefore has rank at least $r$.  On the other hand, every intermediate
adjacency matrix is a principal submatrix of $A(G)$, so its rank is at
most $r(G)=r$.  Hence the rank remains exactly $r$ throughout.  Applying
Lemma~\ref{lem:deletion} successively gives
\[
 \E(G)-\dd(G)\ge\E(H)-\dd(H).
\]
Since $H$ is nonsingular of order $r\ge5$, the nonsingular case already
proved gives
\[
 \E(H)-\dd(H)\ge r-1.
\]

If $r<5$, enlarge $H$ to an induced subgraph $K$ of order five.  Since
$K$ contains the nonsingular principal submatrix $A(H)$, while $A(K)$ is
a principal submatrix of $A(G)$, we have
\[
 r\le r(K)\le r,
\]
so $r(K)=r$.  Deleting the vertices outside $K$ and applying
Lemma~\ref{lem:deletion} successively, followed by
Proposition~\ref{prop:n5}, yields
\[
 \E(G)-\dd(G)
 \ge
 \E(K)-\dd(K)
 \ge
 r-1.
\]
Thus \eqref{eq:main} holds in all cases.
\end{proof}

\section{Equality cases}\label{sec:equality}

We now determine all graphs for which equality holds in
Theorem~\ref{thm:main}.

\begin{lemma}[Integrality in the equality case]\label{lem:equality-integrality}
Suppose
\[
 \E(G)=r(G)+\dd(G)-1.
\]
Then $\E(G)$ is an even integer and $\dd(G)$ is an integer.  If $G$ is
nonsingular of order $n$, then
\[
 a:=\frac{n-1-\dd(G)}2
\]
is a nonnegative integer and
\[
 \frac{\E(G)}2=n-1-a.
\]
\end{lemma}

\begin{proof}
Every adjacency eigenvalue is a real algebraic integer.  Since
$|\lambda_i|$ equals either $\lambda_i$ or $-\lambda_i$, each
$|\lambda_i|$ is again an algebraic integer.  Therefore $\E(G)$ is an
algebraic integer.  Under the
equality hypothesis it is rational, so $\E(G)\in\mathbb Z$.  Moreover,
\[
 \frac{\E(G)}2=\sum_{\lambda_i>0}\lambda_i
\]
is also an algebraic integer and is rational, hence is an integer.  Thus
$\E(G)$ is even, and
\[
 \dd(G)=\E(G)-r(G)+1\in\mathbb Z.
\]
The last assertions are immediate in the nonsingular case.
\end{proof}

Mojallal \cite{MojallalExtremal2026}
characterized all graphs satisfying
$ \E(G)=2\bigl(n-\alpha(G)\bigr)$.

\begin{lemma}\label{lem:KP-equality-nonsingular}
If $G$ is connected, nonsingular, and
$
 \E(G)=2\bigl(n-\alpha(G)\bigr),
$
then $G$ is complete.
\end{lemma}

\begin{proof}
By \cite[Theorem 1.3]{MojallalExtremal2026}, every connected nontrivial
equality graph is either a balanced complete multipartite graph or a graph
$H_r(a,b)$ with $r\ge3$.  A balanced complete multipartite graph with a
part of size greater than one has repeated adjacency rows and is singular;
hence its only nonsingular members are complete graphs.  For $H_r(a,b)$,
if $a>1$ or $b>1$, repeated rows again give singularity.  If $a=b=1$, then
$H_r(1,1)\cong K_2\square K_r$, whose spectrum contains $0$ with
multiplicity $r-1$.  Thus the only connected nonsingular equality graphs
are complete graphs.
\end{proof}

\begin{proposition}\label{prop:singular-strict}
If $G$ has order $n\ge5$ and is singular, then
\[
 \E(G)>r(G)+\dd(G)-1.
\]
\end{proposition}

\begin{proof}
If $r(G)=0$, then $G$ is empty and the assertion is immediate.  Assume
$r:=r(G)>0$.  By Lemma~\ref{lem:pseudodet}, $G$ has an induced nonsingular
subgraph $H$ of order $r$.

Suppose first that $n>5$.  If $r\ge5$, delete the vertices outside $H$.
At least one deletion occurs, and every intermediate graph has at least one edge.
Lemma~\ref{lem:strict-deletion} therefore gives
\[
 \E(G)-\dd(G)>\E(H)-\dd(H)\ge r-1.
\]
If $r<5$, enlarge $H$ to an induced subgraph $K$ of order five.  As in the
proof of Theorem~\ref{thm:main}, $r(K)=r$.  Deleting down to $K$ and using
Proposition~\ref{prop:n5} gives
\[
 \E(G)-\dd(G)>\E(K)-\dd(K)\ge r-1.
\]

It remains to consider $n=5$.  Suppose equality holds.  By
Lemma~\ref{lem:equality-integrality}, $\dd(G)$ is an integer.  Since
$\dd(G)=2m/5$, we have $\dd(G)\in\{0,2,4\}$.  The values $0$ and $4$
correspond to the empty graph and $K_5$, respectively, and neither gives
a singular equality graph.  Hence $\dd(G)=2$.  Since $\E(G)$ is even,
equality forces $r$ to be odd.  Rank one is impossible for a nonzero
symmetric zero-diagonal matrix, so singularity gives $r=3$ and
$\E(G)=4$.  But
\[
 \E(G)\ge2\lambda_1(G)\ge2\dd(G)=4.
\]
Thus equality holds throughout, so $\lambda_1(G)=\dd(G)$ and the all-ones
vector is a Perron eigenvector.  Hence $G$ is $2$-regular.  The only
$2$-regular graph on five vertices is $C_5$, which is nonsingular, a
contradiction.
\end{proof}

The following elementary scalar check is used only in the equality analysis
for $5\le n\le10$, allowing us to avoid the computer verification used in
the proof of \cite[Theorem 2.1]{AkbariDabirianGhasemi2022}.

\begin{lemma}\label{lem:small-equality-scalar}
Let $a\in\{1,2,3\}$, let $k>a$ be an integer, and let
\[
 w\ge\max\{0,2a-k\}.
\]
Put
\[
 Z=w+a^2+k,\qquad
 \tau=\min\left\{a,\frac Z6\right\},
\]
and
\[
 F_3(k,a,w)=
 \log\frac Z2+(k-1)\log\tau-k\log k+k-a-g(w).
\]
Then
\[
 F_3(k,a,w)<0.
\]
\end{lemma}

\begin{proof}
Set $s=\sqrt w$.  On the branch $Z\ge6a$, where $\tau=a$, and on the
branch $Z\le6a$, direct differentiation gives
\[
 \frac{\partial F_3}{\partial s}
 =
 \frac{s\{3-a^2-k-(s-1)^2\}}{Z(1+s)}
\]
on the first branch, and
\[
 \frac{\partial F_3}{\partial s}
 =
 \frac{s\{k^2+k-a^2-(s-k)^2\}}{Z(1+s)}
\]
on the second.  Since $a\in\{1,2,3\}$ and $k>a$, the first branch is
nonincreasing.  On the second branch, the upper critical point is
\[
 s=k+\sqrt{k^2+k-a^2}.
\]
At such a point,
\[
 Z=2k(1+s)>2k(k+1)\ge2(a+1)(a+2)>6a,
\]
so it lies outside the branch $Z\le6a$.  The lower critical point, when
positive, is a local minimum.  Hence, for fixed $a$ and $k$, a maximum can
occur only at the lower endpoint of the allowed $s$-interval or at the
branch boundary $Z=6a$.

We first consider the branch boundary.  There
\[
 s^2=6a-a^2-k
\]
and
\[
 B_a(k):=F_3(k,a,s^2)
 =
 \log3+k\log a-k\log k+k-a-g(s^2).
\]
Treating $k$ temporarily as a continuous variable,
\[
 B_a'(k)
 =
 \log\frac ak+\frac{1}{2(1+s)}.
\]
This derivative is negative throughout the feasible boundary interval.
Indeed, if $a=1$, then $k\ge2$, so
\[
 B_1'(k)\le-\log2+\frac12<-\frac16.
\]
If $a=2$, then for $3\le k\le7$ we have $s\ge1$, and therefore
\[
 B_2'(k)
 \le-\log\frac32+\frac14<-\frac{3}{20};
\]
for $7\le k\le8$,
\[
 B_2'(k)
 \le-\log\frac72+\frac12<0.
\]
If $a=3$, then for $4\le k\le8$ we again have $s\ge1$, and
\[
 B_3'(k)
 \le-\log\frac43+\frac14<-\frac1{28};
\]
for $8\le k\le9$,
\[
 B_3'(k)
 \le-\log\frac83+\frac12<0.
\]
Here all logarithmic comparisons follow immediately from
\eqref{eq:log-lower}.  Thus the largest branch-boundary value occurs at
$k=a+1$.

It remains to check the lower endpoints.  When the lower endpoint is
$w=0$ and lies on the second branch, differentiation with respect to $k$
gives
\[
 \frac{\partial F_3(k,a,0)}{\partial k}
 =
 \frac{k}{a^2+k}
 +
 \log\frac{a^2+k}{6k}.
\]
Writing $t=1+a^2/k$, this becomes
\[
 \frac1t+\log\frac t6.
\]
For the relevant values $k\ge2a$ we have $1<t\le5/2$, and the last
expression is increasing in $t$ and is at most
\[
 \frac25+\log\frac5{12}<0,
\]
again by \eqref{eq:log-lower}.  Hence the largest such endpoint values
occur at $k=2a$, and they are
\[
 1+\log3-4\log2,\qquad
 2-3\log3,\qquad
 3+6\log\frac54-5\log3,
\]
for $a=1,2,3$, respectively.  All are negative, using $3<4$, $\log2>2/3$, $\log3>1$, and
$\log(5/4)<1/4$.

If $k>6a-a^2$, the branch boundary no longer exists and $w=0$ lies on
the first branch.  There
\[
 \frac{\partial F_3(k,a,0)}{\partial k}
 =
 \frac1{a^2+k}+\log\frac ak.
\]
Since $\log(a/k)<a/k-1=-(k-a)/k$ and
$1/(a^2+k)<1/k\le(k-a)/k$, this derivative is negative.  At
$k=6a-a^2$ the point $w=0$ is precisely the branch boundary, whose value
is already negative.  Hence all first-branch lower endpoints with
$k>6a-a^2$ are negative as well.

The only lower endpoints with $w=2a-k>0$ are
\[
 (a,k,w)=(2,3,1),\ (3,4,2),\ (3,5,1).
\]
Their values are
\[
 7\log2-5\log3,
 \qquad
 1-\sqrt2+
 \log\!\left(\frac{1875(1+\sqrt2)}{4096}\right),
 \qquad
 1+\log3-4\log2.
\]
The first is negative because $\log2<7/10$ and $\log3>1$, and the third
is negative because $3<4$ and $\log2>2/3$.  For the middle expression,
use
\[
 \frac75<\sqrt2<\frac32,
 \qquad
 \frac{1875(1+\sqrt2)}{4096}<\frac43,
 \qquad
 \log\frac43<\frac13;
\]
then it is smaller than
\[
 1-\frac75+\frac13=-\frac1{15}.
\]
Thus every lower endpoint is negative.

It remains only to check the largest branch-boundary value for each $a$.
Substituting $k=a+1$ gives
\[
\begin{array}{c|c}
a&\max F_3\\
\hline
1&
\displaystyle
1-\sqrt3+
\log\!\left(\frac{3(1+\sqrt3)}4\right)
\\[2mm]
2&
\displaystyle
1-\sqrt5+
\log\!\left(\frac{8(1+\sqrt5)}9\right)
\\[2mm]
3&
\displaystyle
1-\sqrt5+
\log\!\left(\frac{243(1+\sqrt5)}{256}\right).
\end{array}
\]
These are negative by elementary rational comparisons.  Indeed,
\[
 \frac{43}{25}<\sqrt3<\frac{26}{15},
\]
and the exponential series gives
\[
 e^{18/25}
 >
 1+\frac{18}{25}
 +\frac12\left(\frac{18}{25}\right)^2
 +\frac16\left(\frac{18}{25}\right)^3
 +\frac1{24}\left(\frac{18}{25}\right)^4
 >
 \frac{41}{20}.
\]
Hence
\[
 \frac{3(1+\sqrt3)}4<\frac{41}{20}<e^{18/25}
 <e^{\sqrt3-1}.
\]
Also
\[
 \frac{11}{5}<\sqrt5<\frac94,
\]
and
\[
 e^{6/5}
 >
 1+\frac65+\frac{18}{25}+\frac{36}{125}+\frac{54}{625}
 >
 \frac{13}{4}.
\]
Thus
\[
 \frac{8(1+\sqrt5)}9<3<e^{6/5}<e^{\sqrt5-1}
\]
and
\[
 \frac{243(1+\sqrt5)}{256}
 <\frac{13}{4}
 <e^{6/5}
 <e^{\sqrt5-1}.
\]
Taking logarithms proves that all three displayed maxima are strictly
negative.
\end{proof}

\begin{proposition}\label{prop:connected-equality}
Let $G$ be connected and nonsingular of order $n\ge5$.  If equality holds
in \eqref{eq:main}, then $G\cong K_n$.
\end{proposition}

\begin{proof}
Write $d=\dd(G)$ and let the inertia be $(p,q)$.  By
Lemma~\ref{lem:equality-integrality},
\[
 a=\frac{n-1-d}{2}\in\mathbb Z_{\ge0},
 \qquad
 T:=\frac{\E(G)}2=n-1-a.
\]
As in \eqref{eq:T-pq},
\begin{equation}\label{eq:equality-Tpq}
 T\ge n-\alpha(G)\ge\max\{p,q\}.
\end{equation}
If $a=0$, then $d=n-1$ and $G=K_n$.  Hence assume $a\ge1$.

If $p=1$, then \eqref{eq:inertia-alpha} gives $\alpha(G)=1$, again
$G=K_n$.  If $p=2$, then $q=n-2$ and
\[
 T=n-1-a\ge q
\]
forces $a=1$ and $T=q$.  Thus equality holds throughout
\eqref{eq:equality-Tpq}, so
\begin{equation}\label{lowerbound-alpha}
 \E(G)=2\bigl(n-\alpha(G)\bigr).
\end{equation}

Lemma~\ref{lem:KP-equality-nonsingular} gives $G=K_n$, contrary to $p=2$.

Now assume $p\ge3$ and put $k=p-1$.  From
\eqref{eq:equality-Tpq},
\[
 a\le k,\qquad a\le q-1.
\]
If $a=k$ or $a=q-1$, then $T$ equals $q$ or $p$, respectively, and again
we arrive at \eqref{lowerbound-alpha}.  Lemma~\ref{lem:KP-equality-nonsingular} gives a contradiction.  We may therefore
assume
\begin{equation}\label{eq:equality-interior}
 0<a<\min\{k,q-1\}.
\end{equation}

Suppose first that
\[
 \rho>7.11,
 \qquad
 d>n-2\log n-3.
\]
Then $a<\log n+1$, and, since $\rho\ge d=n-1-2a$,
\[
 \rho>h(a)
 =\max\{7.11,e^{a-1}-1-2a\}.
\]
Repeat the determinant--variance argument of
Section~\ref{sec:reduction} with $T=S$.  The only changes are
\[
 \prod_{i=1}^k x_i\le\left(\frac Rk\right)^k,
 \qquad
 U=k-a,
 \qquad
 V\ge2a-k,
\]
while the estimate
\[
 V>2\rho R-a^2-k
\]
remains strict because $k\ge2$ and
$\sum_i x_i^2<R^2$.  Hence the same definitions
\eqref{eq:wZtau} give
\[
 \log|\det A|\le F(k,a,w)<0
\]
by Lemma~\ref{lem:scalar}, impossible.

Thus an equality graph must lie in at least one of the two ranges from
\cite{AkbariDabirianGhasemi2022}.  Assume first that $\rho\le7.11$.
For $n\ge12$, the proof of
\cite[Theorem 2.1]{AkbariDabirianGhasemi2022} actually yields
\[
 \E(G)\ge n-1+d+
 \frac{9\log|\det A|+(n-11)(d-1)}{11}.
\]
Since $G$ is connected and $d$ is an integer, $d\ge2$, and the last term
is positive.  If $n=11$, equality in that proof would force equality in
the pointwise inequality of their Lemma 2.1 for every $|\lambda_i|$.
Its only equality point in $(0,7.11]$ is $1$, so every
$|\lambda_i|=1$, impossible because eleven numbers from $\{1,-1\}$ cannot
sum to zero.

It remains in this range to consider $5\le n\le10$.  If $d=2$, then $G$
is unicyclic.  Deleting a vertex on its cycle leaves a forest on $n-1$
vertices, so
\[
 \alpha(G)\ge\frac{n-1}{2}.
\]
Equality gives $\E(G)=n+1$, which is even, hence $n$ is odd.  Since
$p+q=n$, we have
\[
 \max\{p,q\}\ge\frac{n+1}{2}.
\]
On the other hand, \eqref{eq:inertia-alpha} gives
\[
 \max\{p,q\}\le n-\alpha(G),
\]
and therefore
\[
 \alpha(G)\le\frac{n-1}{2}.
\]
Combining this with the preceding lower bound yields
\[
 \alpha(G)=\frac{n-1}{2},
 \qquad
 \E(G)=2\bigl(n-\alpha(G)\bigr).
\]
Lemma~\ref{lem:KP-equality-nonsingular} then gives $G=K_3$, contrary to
$n\ge5$.

Hence $d\ge3$.  Then $\rho\ge d\ge3$, while
$a=(n-1-d)/2\in\{1,2,3\}$.  Repeating the equality version of the
determinant--variance reduction with the constant lower bound $h=3$
gives
\[
 \log|\det A|\le F_3(k,a,w)<0
\]
by Lemma~\ref{lem:small-equality-scalar}, again impossible.

Finally, suppose
\[
 \rho>7.11,
 \qquad
 d\le n-2\log n-3.
\]
If $d\le5$, the determinant bound used in
\cite{AkbariDabirianGhasemi2022} gives
\[
 \E(G)\ge n-1+\rho-\log\rho
 >n-1+5
 \ge n-1+d,
\]
since $x-\log x$ is increasing for $x>1$ and
$7.11-\log7.11>5$.

Assume now that $d\ge6$.  In the proof of
\cite[Theorem 3.3]{AkbariDabirianGhasemi2022}, the function
\[
 \phi(x)=\sqrt{x}-\log\sqrt{x}
\]
is used.  Since
\[
 \phi'(x)
 =
 \frac{\sqrt{x}-1}{2x}>0
 \qquad(x>1),
\]
the function $\phi$ is strictly increasing on $(1,\infty)$.  Hence
\[
 \phi(nd-n+1)>\phi(nd-n).
\]
The proof of \cite[Theorem 3.3]{AkbariDabirianGhasemi2022} shows that
\[
 \phi(nd-n)-d\ge0,
\]
and therefore
\[
 \phi(nd-n+1)-d>0.
\]
Consequently, the nonnegative correction term appearing in their energy
estimate is in fact strictly positive, and their argument yields
\[
 \E(G)>n-1+d.
\]
Thus equality cannot occur in this range.
\end{proof}

\begin{proof}[Proof of Theorem~\ref{thm:equality}]
The graphs listed in the theorem do attain equality.  Indeed,
\[
 \operatorname{Spec}(K_n)=\{n-1,-1^{[n-1]}\},
\]
so
\[
 \E(K_n)=2(n-1)=n+(n-1)-1.
\]
If $G=tK_2$, then $n=2t$, $r(G)=n$, $\dd(G)=1$, and
$\E(G)=n$, so equality again holds.

Conversely, suppose equality holds.  Proposition~\ref{prop:singular-strict}
shows that $G$ is nonsingular.  If $G$ is connected,
Proposition~\ref{prop:connected-equality} gives $G\cong K_n$.  If $G$ is disconnected, Proposition~\ref{prop:connected-equality}
together with the equality statement of Lemma~\ref{lem:disconnected}
gives
\[
 G\cong\frac n2K_2.
\]
This completes the characterization.
\end{proof}

\section{Consequences}\label{sec:consequences}

We first record the two consequences involving the first Zagreb index.  These
settle the conjectured lower bounds of Das and Ghalavand
\cite{DasGhalavand2025}.  Recall that
\[
 M_1(G)=\sum_{v\in V(G)}d_v^2.
\]

\begin{corollary}\label{cor:zagreb}
Let $G$ be a nonsingular graph of order $n$.  Then
\[
 \E(G)\ge\frac{M_1(G)}{m}
 \qquad\text{and}\qquad
 \E(G)\ge\frac{M_1(G)}{2m}+\frac{2m}{n}.
\]
Equality in either inequality holds if and only if $G\cong K_n$.
\end{corollary}

\begin{proof}
Suppose first that $n\ge5$.  De Caen's inequality \cite{deCaen1998} gives
\[
 M_1(G)\le
 m\left(n-2+\frac{2m}{n-1}\right).
\]
Consequently,
\[
 \frac{M_1(G)}m
 \le n-2+\frac{2m}{n-1}
 \le n-1+\frac{2m}{n}
 =n-1+\dd(G)
 \le \E(G),
\]
where the middle inequality follows from $2m\le n(n-1)$ and the last one
from Theorem~\ref{thm:main}.  Moreover, Cauchy--Schwarz gives
\[
 M_1(G)\ge\frac{(2m)^2}{n},
\]
and hence
\[
 \frac{M_1(G)}m
 \ge
 \frac{M_1(G)}{2m}+\frac{2m}{n}.
\]
Thus both inequalities follow.  Equality in either bound forces equality
throughout the preceding chain, and in particular
$2m=n(n-1)$.  Hence $G\cong K_n$.

For $n\le4$, the nonsingular graphs are, up to isomorphism,
$K_2$, $K_3$, $2K_2$, $P_4$, the paw, and $K_4$, and the assertions
follow by direct calculation.
\end{proof}

We next record structural consequences of the rank formulation.  The clique
number $\omega(G)$ is the maximum order of a complete subgraph, and the
induced matching number $\nu_{\rm ind}(G)$ is the maximum size of a
matching whose endpoints induce exactly the edges of the matching.  For a
graph without isolated vertices, the total domination number $\gamma_t(G)$
is the minimum cardinality of a set $S\subseteq V(G)$ such that every
vertex of $G$ has a neighbor in $S$.  For a connected graph,
$\operatorname{diam}(G)$ denotes its diameter.

\begin{corollary}\label{cor:rank-consequences}
Let $G$ be a graph of order $n\ge5$.  Then
\[
 \E(G)\ge \omega(G)+\dd(G)-1
\]
and
\[
 \E(G)\ge 2\nu_{\rm ind}(G)+\dd(G)-1.
\]
If $G$ has no isolated vertices, then
\[
 \E(G)\ge \gamma_t(G)+\dd(G)-1.
\]
If $G$ is connected, then
\[
 \E(G)>\operatorname{diam}(G)+\dd(G)-1.
\]

Equality in the clique-number bound holds exactly for complete graphs,
whereas equality in either the induced-matching or total-domination bound
holds exactly for perfect matchings.
\end{corollary}

\begin{proof}
Each inequality follows from Theorem~\ref{thm:main} together with a
corresponding lower bound for $r(G)$.

An induced $K_{\omega(G)}$ gives a nonsingular principal submatrix of
order $\omega(G)$, and hence
\[
 r(G)\ge\omega(G).
\]
Likewise, an induced matching of size $\nu_{\rm ind}(G)$ gives a principal
submatrix equal to the adjacency matrix of
$\nu_{\rm ind}(G)K_2$, whose rank is $2\nu_{\rm ind}(G)$.  Thus
\[
 r(G)\ge2\nu_{\rm ind}(G).
\]

Suppose that $G$ has no isolated vertices.  For each connected component
$H$ of $G$ that is not complete, Theorem~3.1 of
\cite{AbiadAkbariFakharanMehdizadeh2023}, applied with the eigenvalue
$\lambda=0$, gives
\[
 \gamma_t(H)\le |V(H)|-m_H(0)=r(H).
\]
If $H=K_s$ for some $s\ge2$, then
\[
 \gamma_t(H)=2\le s=r(H).
\]
Thus $\gamma_t(H)\le r(H)$ for every connected component $H$ of $G$.
Since both total domination number and adjacency rank are additive over
components,
\[
 \gamma_t(G)\le r(G).
\]
The asserted total-domination bound now follows from
Theorem~\ref{thm:main}.

Finally, suppose that $G$ is connected and put
$D=\operatorname{diam}(G)$.  A shortest path between two vertices at
distance $D$ is an induced $P_{D+1}$.  Since
\[
 r(P_{D+1})\ge D,
\]
we obtain $r(G)\ge D$, and therefore
\[
 \E(G)\ge D+\dd(G)-1.
\]
If equality held here, then equality would also hold in
Theorem~\ref{thm:main}.  Since $G$ is connected,
Theorem~\ref{thm:equality} would give $G\cong K_n$, but
\[
 r(K_n)=n>\operatorname{diam}(K_n)=1
\]
for $n\ge5$.  Hence the diameter inequality is strict.

For the equality statements in the first three bounds, equality forces
equality in Theorem~\ref{thm:main}.  Thus $G$ is either $K_n$ or a perfect
matching.  For $K_n$,
\[
 r(K_n)=\omega(K_n)=n,
\]
whereas $\nu_{\rm ind}(K_n)=1$ and $\gamma_t(K_n)=2$.  For a perfect
matching,
\[
 r(G)=2\nu_{\rm ind}(G)=\gamma_t(G)=n.
\]
The equality characterizations follow.
\end{proof}

\section*{Acknowledgements}
The author acknowledges the use of AI tools in the preparation of this manuscript.

\bibliographystyle{amsplain}
\bibliography{references}

\bigskip

\noindent
Seyed Ahmad Mojallal, Email:
\texttt{seyed\_ahmad\_mojallal@sfu.ca},
\texttt{ahmad\_mojalal@yahoo.com}

\noindent
Department of Mathematics, Simon Fraser University, Burnaby, BC, Canada

\end{document}